\documentclass{amsart}
\pdfoutput=0
\usepackage{epsfig}
\usepackage[T1]{fontenc}    
\usepackage{ae,aecompl}     
\usepackage{amsmath}        
\usepackage{amssymb}
\usepackage{amsthm}
\usepackage{amsfonts}
\usepackage{amscd}
\usepackage{color}
\usepackage{mathrsfs}       
\usepackage{bbm}            
\usepackage{bm}             
\usepackage{url}            
\usepackage{paralist}       
\usepackage{graphics}       
\usepackage{epic}          
\usepackage[rel]{overpic}   
\usepackage{psfrag}
\usepackage{xspace}         

\input xypic
\xyoption{all}
\xyoption{arc}

\newtheorem{thm}{Theorem}[section]
\newtheorem{cor}[thm]{Corollary}

\newtheorem{prop}[thm]{Proposition}

\def\Do{{\mathsf D}}      
\def\Up{{\mathsf U}}      
\def\ADG{{\mathscr A}}    
\def\SOT{{\mathcal T}}    
\def\Ass{{\mathsf{Asso}}}  
\def\Perm{{\mathsf{Perm}}}
\def\R{{\mathbb R}}       

\newcommand{\vect}[1]{\overrightarrow{#1}}
\def\Stab{{\textnormal{Stab}}}

\def\D{{\mathcal D}}

\begin{document}


\title[Center of gravity of the associahedron]{The centers of gravity of the associahedron and of the permutahedron are the same}

\author[C. Hohlweg]{Christophe~Hohlweg}
\address[Christophe Hohlweg]{Universit\'e du Qu\'ebec \`a Montr\'eal\\
LaCIM et D\'epartement de Math\'ematiques\\ CP 8888 Succ. Centre-Ville\\
Montr\'eal, Qu\'ebec, H3C 3P8\\ CANADA}
\email{hohlweg.christophe@uqam.ca}
\urladdr{http://www.lacim.uqam.ca/\~{}hohlweg}

\author[J. Lortie]{Jonathan~Lortie}
\address[Jonathan Lortie]{Universit\'e du Qu\'ebec \`a Montr\'eal\\
LaCIM et D\'epartement de Math\'ematiques\\ CP 8888 Succ. Centre-Ville\\
Montr\'eal, Qu\'ebec, H3C 3P8\\ CANADA}
\email{lortie.jonathan@courrier.uqam.ca}

\author[A. Raymond]{Annie~Raymond}
\address[Annie Raymond]{Berlin Mathematical School\\
Strasse des 17. Juni 136\\
Berlin, 10623, Germany}
\email{raymond@math.tu-berlin.de}

\date{\today}

\thanks{$^*$ This work is supported by FQRNT and NSERC.
It is the result of a summer undergraduate research internship
supported by LaCIM}

\begin{abstract}
\noindent In this article, we show that Loday's realization of the
associahedron has the the same center of gravity than the
permutahedron. This proves an observation made by F.~Chapoton.

We also prove that this result holds for the associahedron and the
cyclohedron as realized by the first author and C.~Lange.
 \end{abstract}

 \maketitle


\section{Introduction.}\label{se:Intro}

In 1963, J.~Stasheff discovered the associahedron~\cite{stasheff,stasheff2}, a polytope of great importance  in algebraic topology.
 The associahedron in $\mathbb R^n$ is a simple
$n-1$-dimensional convex polytope. The classical realization of the associahedron given by
 S.~Shnider and S.~Sternberg in \cite{shnider_sternberg} was
 completed by J.~L.~Loday in 2004~\cite{loday}. Loday gave a
 combinatorial algorithm to compute the integer coordinates of the
 vertices of the associahedron, and showed that it can be obtained
 naturally from the classical permutahedron of dimension $n-1$.
 F.~Chapoton observed that the centers of gravity of the
 associahedron and of the permutahedron are the same \cite[Section 2.11]{loday}.
   As far as we know, this property of Loday's realization has never been proved.

\smallskip

In 2007, the first author and C.~Lange gave a family of realizations
of the associahedron that contains the classical realization of the
associahedron. Each of these realizations is also obtained naturally
from the classical permutahedron \cite{realisation1}. They
conjectured that for any of these realizations, the center of
gravity coincide with the center of gravity of the permutahedron. In
this article, we prove this conjecture to be true.

\smallskip

The associahedron fits in a larger family of polytopes, {\em
generalized associahedra}, introduced by S.~Fomin and A.~Zelevinsky
in \cite{fomin_zelevinsky} within the framework of cluster algebras
(see \cite{chapoton_fomin_zelevinsky,realisation2} for their
realizations).

In 1994, R.~Bott and C.~Taubes discovered the
cyclohedron~\cite{bott_taubes} in connection with knot theory. It
was rediscovered independently by R. Simion \cite{simion}. In
\cite{realisation1}, the first author and C.~Lange also gave a
family of realizations for the cyclohedron, starting with the
permutahedron of type $B$.

We also show that the centers of gravity of the cyclohedron and of
the permutahedron of type $B$ are the same.

The article is organized as follows. In \S\ref{se:1}, we first
recall the realization of the permutahedron and how to compute its
center of gravity. Then we compute the center of gravity of Loday's
realization of the associahedron. In order to do this, we partition
its vertices into isometry classes of triangulations, which
parameterize  the vertices, and we show that the center of gravity
for each of  those classes is the center of gravity of the
 permutahedron.

 In \S\ref{se:2}, we show that the computation of the center of
 gravity of any of the realizations given by the first author and
 C.~Lange is reduced to the computation of the center of gravity of the classical
 realization of the associahedron. We do the same for the cyclohedron in \S\ref{se:3}.

We are grateful to Carsten Lange for allowing us to use some of the pictures he made in~\cite{realisation1}.

\section{Center of gravity of the classical permutahedron and
associahedron}\label{se:1}

\subsection{The permutahedron}
Let $S_n$ be the symmetric group acting on the set
$[n]=\{1,2,\dots,n\}$. The {\em permutahedron} $\Perm(S_n)$ is the
classical $n-1$-dimensional simple convex polytope defined as the
convex hull of the points
$$
M(\sigma)=(\sigma(1),\sigma(2),\dots, \sigma (n))\in\mathbb R^n,\qquad \forall \sigma\in S_n.
$$
The {\em center of gravity} (or {\em isobarycenter}) is the unique point $G$ of $\mathbb R^n$  such that
$$
\sum_{\sigma\in S_n} \vect{GM(\sigma)}=\vect 0.
$$
Since the permutation $w_0:i\mapsto n+1-i$ preserves $\Perm(S_n)$,
we see, by sending $M(\sigma)$ to
$$
M(w_0\sigma)=(n+1-\sigma(1),n+1-\sigma(2),\dots, n+1-\sigma (n)),
$$
that the center of gravity is  $
G=(\frac{n+1}{2},\frac{n+1}{2},\dots,\frac{n+1}{2}). $

\subsection{Loday's realization}
We present here the realization of the associahedron given by
J.~L.~Loday \cite{loday}. However, instead of using planar binary
trees, we use triangulations of a regular polygon to parameterize
the vertices of the associahedron (see \cite[Remark
1.2]{realisation1}).

\subsubsection{Triangulations of a regular polygon} Let $P$ be a
regular $(n+2)$-gon in the Euclidean plane with vertices
$A_0,A_1,\dots,A_{n+1}$ in counterclockwise direction. A {\em
triangulation of $P$} is a set of $n$ noncrossing diagonals of $P$.

Let us be more explicit. A {\em triangle
of $P$} is a triangle whose vertices are vertices of $P$. Therefore
a side of a triangle of $P$ is either an edge or a diagonal of $P$.
A  triangulation of $P$ is then a collection of $n$ distinct
triangles of $P$ with noncrossing sides. Any of the triangles in $T$ can be described as  $A_i A_j A_k$ with $0\leq i<j<k\leq n+1$. Each $1\leq j\leq n$ corresponds to a unique triangle $\Delta_j(T)$ in $T$ because the sides of  triangles in $T$ are noncrossing.

Therefore we write $T=\{\Delta_1(T),\dots, \Delta_n(T)\}$ for a
triangulation $T$, where $\Delta_j(T)$ is the unique triangle in $T$
with vertex $A_j$ and the two other vertices  $A_i$ and $A_k$
satisfying the inequation $0\leq i<j<k\leq n+1$.

Denote by $\SOT_{n+2}$ the set of triangulations of $P$.

\subsubsection{Loday's realization of the associahedron}

Let $T$ be a triangulation of $P$. The {\em weight} $\delta_j(T)$ of
the triangle $\Delta_j(T)=A_i A_jA_k$, where $i<j<k$, is the positive
number
$$
\delta_j(T)=(j-i)(k-j).
$$
The weight $\delta_j(T)$ of $\Delta_j(T)$ represents the product of the number of
boundary edges of $P$ between $A_i$ and $A_j$ passing through vertices
indexed by smaller numbers than $j$ with the number of boundary
edges of $P$ between $A_j$ and $A_k$ passing through vertices indexed
by larger numbers than $j$.

The {\em classical associahedron} $\Ass(S_n)$ is obtained as the
convex hull of the points
$$
M(T)=(\delta_1(T),\delta_2(T),\dots, \delta_n(T))\in \mathbb
R^n,\quad\forall T\in\SOT_{n+2}.
$$
We are now able to state our first result.

\begin{thm}\label{thm:Main} The center of gravity of $\Ass(S_n)$ is $G=(\frac{n+1}{2},\frac{n+1}{2},\dots,\frac{n+1}{2})$.
 \end{thm}

In order to prove this theorem, we need to study closely a certain partition of the vertices of $P$.

\subsection{Isometry classes of triangulations}\label{se:centergravity}
As $P$ is a regular $(n+2)$-gon, its isometry group is the dihedral
group $\D_{n+2}$ of order $2(n+2)$. So $\D_{n+2}$ acts on the set
$\SOT_{n+2}$ of all triangulations of $P$: for $f\in\D_{n+2}$ and
$T\in\SOT_{n+2}$, we have $f\cdot T\in\SOT_{n+2}$. We denote by
$\mathcal O (T)$ the orbit of $T\in\SOT_{n+2}$ under the action of
$\D_{n+2}$.

We know that $G$ is the center of gravity of $\Ass(S_n)$ if and only if
$$
\sum_{T\in\SOT_{n+2}} \vect{GM(T)} =\vect 0.
$$
As the orbits of the action of $\D_{n+2}$ on $\SOT_{n+2}$ form a partition of the set $\SOT_{n+2}$, it is sufficient to compute
$$
\sum_{T\in\mathcal O} \vect{GM(T)}
$$
for any orbit $\mathcal O$. The following key observation implies
directly Theorem~\ref{thm:Main}.

\begin{thm}\label{thm:key} Let $\mathcal O$ be an orbit of the action of $\D_{n+2}$ on $\SOT_{n+2}$, then $G$ is the center of gravity of $\{M(T)\,|\, T\in\mathcal O\}$. In particular, $
\sum_{T\in\mathcal O} \vect{GM(T)}=\vect 0.
$
\end{thm}

Before proving this theorem, we need to prove the following result.

\begin{prop}\label{prop:canonique} Let $T\in\SOT_{n+2}$ and $j\in [n]$, then
$\displaystyle{\sum_{f\in \D_{n+2}} \delta_j(f\cdot T) =
(n+1)(n+2)}$.
\end{prop}

\begin{proof}
 We prove this proposition by induction on $j\in [n]$. For any triangulation $T'$, we denote by
 $a_j(T')<j<b_j(T')$ the indices of the vertices of $\Delta_j(T')$. Let $H$ be the group of  rotations
 in $\D_{n+2}$. It is well-known that for any reflection $s\in \D_{n+2}$, the classes $H$ and
$sH$ form a partition of $\D_{n+2}$ and that $|H|=n+2$.  We consider
also the unique reflection $s_k\in\D_{n+2}$  which maps $A_x$ to
$A_{n+3+k-x}$, where the values of the indices are taken in modulo
$n+2$. In particular, $s_k(A_0)=A_{n+3+k}=A_{k+1}$, $s_k(A_1)=A_k$,
$s_k(A_{k+1})=A_{n+2}=A_0$,  and so on.

\smallskip

\noindent {\bf Basic step $j=1$:}  We know  that $a_1(T')=0$ for any
triangulation $T'$, hence the weight of $\Delta_1(T')$ is
$\delta_1(T')=(1-0)(b_1(T')-1)=b_1(T')-1$.

The reflection $s_0\in \D_{n+2}$ maps $A_x$
to $A_{n+3-x}$ (where $A_{n+2}=A_0$ and $A_{n+3}=A_1$). In other
words, $s_0(A_0)=A_1$ and $s_0(\Delta_1(T'))$ is a triangle in
$s_0\cdot T'$. Since
$$
s_0(\Delta_1( T'))= s_0(A_0A_1A_{b_1(T')})= A_0A_1A_{n+3-b_1(T')}
$$
and $0<1<n+3-b_1(T')$, $s_0(\Delta_1(T'))$ has to be
$\Delta_1(s_0\cdot T')$. In consequence, we obtain that
$$
\delta_1(T')+\delta_1(s_0\cdot T')= (b_1(T')-1)+(n+3-b_1(T')-1)=n+1,
$$
for any triangulation $T'$. Therefore
$$
\sum_{f\in \D_{n+2}} \delta_1(f\cdot T) = \sum_{g\in H}\big(
(\delta_1(g\cdot T)+\delta_1(s_0\cdot (g\cdot T))\big)= |H|
(n+1)=(n+1)(n+2),
$$
proving the initial case of the induction.

\smallskip

\noindent {\bf Inductive step:} Assume that, for a given $1\leq j<n$, we
have
$$
\sum_{f\in \D_{n+2}}\delta_j(f\cdot T) = (n+1)(n+2).
$$
We will show that
$$
\sum_{f\in \D_{n+2}}\delta_{j+1}(f\cdot T) = \sum_{f\in
\D_{n+2}}\delta_j(f\cdot T).
$$
Let $r\in H\subseteq \D_{n+2}$ be the unique rotation mapping
$A_{j+1}$ to $A_{j}$. In particular, $r(A_0)=A_{n+1}$. Let $T'$ be a
triangulation of $P$. We have two cases:

\smallskip

\noindent {\bf Case 1.} If $a_{j+1}(T')>0$ then
$a_{j+1}(T')-1<j<b_{j+1}(T')-1$ are the indices of the vertices of
the triangle $r(\Delta_{j+1}(T'))$ in $r\cdot T'$. Therefore, by
unicity, $r(\Delta_{j+1}(T'))$ must be $\Delta_j(r\cdot T')$. Thus
 \begin{eqnarray*}
   \delta_{j+1}(T')&=&(b_{j+1}(T')-(j+1))(j+1-a_{j+1}(T'))\\
   &=&\big((b_{j+1}(T')-1)-j\big)(j-(a_{i+1}(T')-1))\\
   &=&\delta_j(r\cdot T').
\end{eqnarray*}
In other words:
\begin{eqnarray}\label{equ:1}
\sum_{{f\in \D_{n+2},\atop a_{j+1}(f\cdot T)\not =
0}}\delta_{j+1}(f\cdot T) & =& \sum_{{f\in \D_{n+2},\atop
a_{j+1}(f\cdot T)\not = 0}}\delta_j(r\cdot(f\cdot T))\\\nonumber &
=& \sum_{{g\in \D_{n+2},\atop b_{j}(g\cdot T)\not =
n+1}}\delta_j(g\cdot T).
\end{eqnarray}

\smallskip

\noindent {\bf Case 2.} If $a_{j+1}(T')=0$, then
$j<b_{j+1}(T')-1<n+1$ are the indices of the vertices of
$r(\Delta_{j+1}(T'))$, which is therefore not $\Delta_j(r\cdot T')$:
it is $\Delta_{b_{j+1}(T')-1}(r\cdot T')$. To handle this, we need
to use the reflections $s_j$ and $s_{j-2}$.

On one hand,  observe that $j+1<n+3+j-b_{j+1}(T')$ because
$b_{j+1}(T')<n+1$.
 Therefore
$$
s_j(\Delta_{j+1}(T'))=A_{j+1}A_0
A_{n+3+j-b_{j+1}(T')}=\Delta_{j+1}(s_j\cdot T').
$$
Hence
\begin{eqnarray*}
\delta_{j+1}(T')+\delta_{j+1}(s_j\cdot
T')&=&(j+1)(b_{j+1}(T')-(j+1))\\
&&+(j+1)(n+3+j-b_{j+1}(T')-(j+1))\\
&=&(j+1)(n+1-j).
\end{eqnarray*}

On the other hand, consider the triangle $\Delta_j(r\cdot T')$ in
$r\cdot T'$. Since
$$
r(\Delta_{j+1}(T'))=A_{j}A_{b_{j+1}(T')-1}A_{n+1}=\Delta_{b_{j+1}(T')-1}(r\cdot T')
$$
is in $r\cdot T'$, $[j,n+1]$ is a diagonal in $r\cdot T'$. Hence
$b_j(r\cdot T')=n+1$. Thus $\Delta_j(r\cdot T')=A_{a_j(r\cdot
T')}A_j A_{n+1}$ and  $\delta_j(r\cdot T')=(j-a_j(r\cdot
T'))(n+1-j)$. We have $s_{j-2}(A_j)=A_{n+1}$, $s_{j-2}(A_{n+2})=A_j$
and $s_{j-2}(A_{a_j(r\cdot T')})=A_{n+1+j-a_j(r\cdot
T')}=A_{j-a_j(r\cdot T')-1}$ since $a_j(r\cdot T')<j$. Therefore
$s_{j-2}(\Delta_j(r\cdot T'))=A_{j-a_j(r\cdot
T')-1}A_jA_{n+1}=\Delta_j(s_{j-2}r\cdot T')$ and
$\delta_j(s_{j-2}r\cdot T')=(a_j(r\cdot T')+1)(n+1-j)$.  Finally we
obtain that
 \begin{eqnarray*}
 \delta_{j}(r\cdot T')+\delta_{j}(s_{j-2}r\cdot
T')&=&(j-a_j(r\cdot T'))(n+1-j)+(a_j(r\cdot
T')+1)(n+1-j)\\
&=&(j+1)(n+1-j).
 \end{eqnarray*}

Since  $\{H,s_k H\}$ forms a partition of $\D_{n+2}$ for any $k$,
we have
\begin{eqnarray}\label{equ:2}
 \sum_{{f\in \D_{n+2},\atop a_{j+1}(f\cdot T)=0}}\delta_{j+1}(f\cdot T) & =& \sum_{{f\in H,\atop a_{j+1}(f\cdot T)=0}}\big(\delta_{j+1}(f\cdot T) +\delta_{j+1}(s_j f\cdot T)\big)\\ \nonumber
 &=& \sum_{{f\in H,\atop  a_{j+1}(f\cdot T)=0}} (j+1)(n+1-j)\\ \nonumber
 &=& \sum_{{rf\in H,\atop  b_{j}(rf\cdot T)=n+1}}\big(\delta_{j}(rf\cdot T) +\delta_{j}(s_{j-2} rf\cdot T)\big),\ \textrm{since }r\in H\\ \nonumber
 &=& \sum_{{g\in H,\atop  b_{j}(g\cdot T)=n+1}}\delta_{j}(g\cdot T).
\end{eqnarray}

\smallskip

\noindent We conclude the induction by adding
Equations~(\ref{equ:1}) and (\ref{equ:2}).
\end{proof}

\begin{proof}[Proof of Theorem~\ref{thm:key}] We have to prove that
$$
\vect u=\sum_{T'\in\mathcal O(T)} \vect{GM(T')}=\vect 0.
$$
Denote by $\Stab(T')=\{f\in\D_{n+2}\,|\, f\cdot T'=T'\}$ the stabilizer of $T'$, then
$$
\sum_{f\in \D_{n+2}} M(f\cdot T) = \sum_{T'\in \mathcal O(T)} |\Stab(T')| M(T').
$$
Since $T'\in\mathcal O(T)$, $|\Stab(T')|=|\Stab(T)|=\frac{2(n+2)}{|\mathcal O(T)|}$, we have
$$
\sum_{f\in \D_{n+2}} M(f\cdot T) = \frac{2(n+2)}{|\mathcal O(T)|} \sum_{T'\in \mathcal O(T)} M(T') .
$$
Therefore by Proposition~\ref{prop:canonique} we have for any $i\in [n]$
\begin{equation}\label{equ:3}
\sum_{T'\in \mathcal O(T)} \delta_i(T')= \frac{|\mathcal
O(T)|}{2(n+2)}(n+1)(n+2)=\frac{|\mathcal O(T)|(n+1)}{2}.
\end{equation}

Denote by $O$ the point of origin of $\mathbb R^n$. Then
$\vect{OM}=M$ for any point $M$ of $\mathbb R^n$. By Chasles'
relation we have finally
 $$
\vect u=\sum_{T'\in\mathcal O(T)} \vect{GM(T')}= \sum_{T'\in\mathcal O(T)} (M(T')-G) =\sum_{T'\in\mathcal O(T)} M(T') - |\mathcal O(T)| G.
$$
So the $i^{th}$ coordinate of $\vect u$ is $ \sum_{T'\in \mathcal
O(T)} \delta_i(T')- \frac{|\mathcal O(T)|(n+1)}{2}=0 $, hence $\vect
u =\vect 0$ by (\ref{equ:3}).
\end{proof}

\section{Center of gravity of generalized associahedra of type $A$ and $B$}\label{se:2}

\subsection{Realizations of associahedra} As a Coxeter group (of type $A$), $S_n$ is generated by the simple
transpositions $\tau_i=(i,\, i+1)$, $i\in [n-1]$. The Coxeter graph
$\Gamma_{n-1}$ is then
\begin{figure}[h]
      \psfrag{t1}{$\tau_{1}$}
      \psfrag{t2}{$\tau_{2}$}
      \psfrag{t3}{$\tau_{3}$}
      \psfrag{tn}{$\tau_{n-1}$}
      \psfrag{dots}{$\ldots$}
      \begin{center}
            \includegraphics[width=8cm]{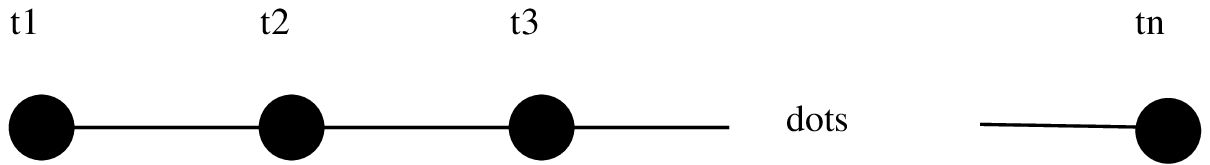}
      \end{center}
\end{figure}

Let $\ADG$ be an orientation of $\Gamma_{n-1}$. We distinguish
between {\em up} and {\em down} elements of $[n]$~: an element $i\in
[n]$ is {\em up} if the edge $\{\tau_{i-1}, \tau_i\}$ is directed
from $\tau_i$ to $\tau_{i-1}$ and {\em down} otherwise (we set $1$
and $n$ to be down). Let $\Do_\ADG$ be the set of down elements and
let $\Up_\ADG$ be the set of up elements (possibly empty).

The notion of up and down induces a labeling of the $(n+2)$-gon $P$
as follows. Label $A_0$ by $0$. Then the vertices of $P$ are, in
counterclockwise direction,  labeled by the down elements in
increasing order, then by $n+1$, and finally by the up elements in
decreasing order. An example is given in
Figure~\ref{fig:example_labelling}.
\begin{figure}[h]
      \psfrag{0}{$0$}
      \psfrag{1}{$1$}
      \psfrag{2}{$2$}
      \psfrag{3}{$3$}
      \psfrag{4}{$4$}
      \psfrag{5}{$5$}
      \psfrag{6}{$6$}
      \psfrag{s1}{$\tau_{1}$}
      \psfrag{s2}{$\tau_{2}$}
      \psfrag{s3}{$\tau_{3}$}
      \psfrag{s4}{$\tau_{4}$}
      \begin{center}
      \begin{minipage}{0.95\linewidth}
         \begin{center}
            \includegraphics[height=6cm]{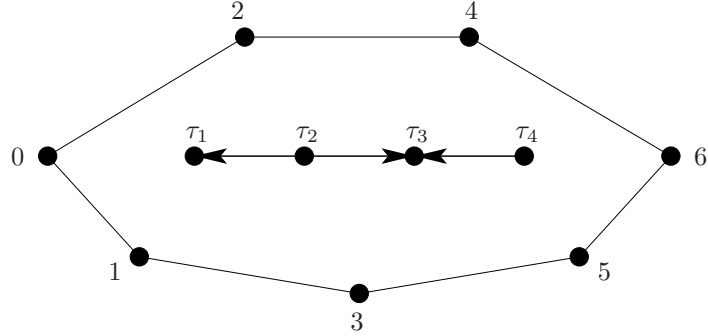}
         \end{center}
         \caption[]{A labeling of a heptagon that corresponds to
                    the orientation $\ADG$ of $\Gamma_{4}$ shown inside
                    the heptagon. We have $\Do_{\ADG} = \{1, 3, 5 \}$ and
                     $\Up_{\ADG} = \{ 2,4\}$.}
         \label{fig:example_labelling}
      \end{minipage}
      \end{center}
\end{figure}

We recall here a construction due to Hohlweg and
Lange~\cite{realisation1}. Consider  $P$ labeled according to a
fixed orientation~$\ADG$ of $\Gamma_{n-1}$. For each $l\in [n]$ and
any triangulation $T$ of $P$, there is a unique triangle
$\Delta^\ADG_l(T)$ whose vertices are labeled by $k<l<m$.  Now,
count the number of edges of $P$ between $i$ and $k$, whose vertices
are labeled by smaller numbers than $l$. Then multiply it by the
number of edges of $P$ between $l$ and $m$, whose vertices are
labeled by greater numbers than $l$. The result $\omega_l^\ADG(T)$
is called the {\em weight} of $\Delta_l^\ADG(T)$. The injective map
\begin{align*}
   M_{\ADG}: \SOT_{n+2} &\longrightarrow  \R^n \\
                   T          &\longmapsto      (x^\ADG_1(T),x^\ADG_2(T),\dots,x^\ADG_n(T))
\end{align*}
that assigns explicit coordinates to a triangulation is defined as follows:
\[
  x^\ADG_j(T) := \begin{cases}
            \omega_j^\ADG  (T)   & \textrm{if } j\in\Do_\ADG\\
            n+1-\omega_j^\ADG(T) & \textrm{if } j\in\Up_\ADG.
         \end{cases}
\]

Hohlweg and Lange showed that the convex hull $\Ass_\ADG(S_n)$
of~$\{M_{\ADG}(T)\,|\,T\in \SOT_{n+2}\}$ is a realization of the
associahedron with integer coordinates \cite[Theorem
1.1]{realisation1}. Observe that if the orientation $\ADG$ is {\em
canonic}, that is, if $\Up_\ADG=\emptyset$, then
$\Ass_\ADG(S_n)=\Ass(S_n)$.

The key is now to observe that the weight of $\Delta_{j}^\ADG(T)$ in
$T$ is precisely the weight of $\Delta_j (T')$ where $T'$ is a
triangulation in the orbit of $T$ under the action of $\D_{n+2}$, as
stated in the next proposition.

\begin{prop}\label{prop:weight} Let $\ADG$ be an orientation of $\Gamma_{n-1}$. Let $j\in [n]$
and let $A_l$ be the vertex of $P$ labeled by $j$. There is an
isometry $r_j^\ADG\in \mathcal D_{n+2}$ such that:
\begin{enumerate}
\item[(i)] $r_j^\ADG(A_l)=A_j$;

\item[(ii)] the label of the vertex $A_k$ is smaller than $j$ if and
only if the index $i$ of the vertex $A_i=r_j^\ADG(A_k)$ is smaller
than $j$.

\end{enumerate}
Moreover, for any triangulation $T$ of $P$ we have
$\omega_j^\ADG(T)=\delta_j(r_j^\ADG\cdot T).$
\end{prop}
\begin{proof} If $\ADG$ is the canonical orientation, then $r_j^\ADG$ is
the identity, and the proposition is straightforward. In the
following proof, we suppose therefore that $\Up_\ADG\not=\emptyset$.

\smallskip

\noindent Case 1: Assume that $j\in\Do_\ADG$. Let $\alpha$ be the
greatest up element smaller than $j$ and let $A_{\alpha+1}$ be the
vertex of $P$ labeled by $\alpha$. Then by construction of the
labeling, $A_{\alpha}$ is labeled by a larger number than $j$, and
$[A_{\alpha},A_{\alpha+1}]$ is the
 unique edge  of $P$ such that $A_{\alpha+1}$ is labeled by a
smaller number than $j$. Denote by $\Lambda_\ADG$ the path from
$A_l$ to $A_{\alpha+1}$ passing through vertices of $P$ labeled by
smaller numbers than $j$. This is the path going from $A_l$ to
$A_{\alpha+1}$ in clockwise direction on the boundary of $P$.

By construction, $A_k\in \Lambda_\ADG$ if and only if the label of
$A_k$ is smaller than $j$. In other words, the path $\Lambda_\ADG$
consists of {\em all} vertices of $P$ labeled by smaller numbers
than $j$. Therefore the cardinality of $\Lambda_\ADG$ is $j+1$.

Consider $r_j^\ADG$ to be the rotation mapping $A_l$ to $A_j$.
Recall that a rotation is an isometry preserving the orientation of
the plane. Then the path $\Lambda_\ADG$, which is obtained by
walking on the boundary of $P$ from $A_l$ to $A_{\alpha+1}$ in
clockwise direction, is sent to the path $\Lambda$ obtained by
walking on the boundary of $P$ in clockwise direction from $A_j$ and
going through $j+1=|\Lambda_\ADG|$ vertices of $P$. Therefore
$\Lambda=\{A_0,A_1,\dots, A_j\}$, thus proving the first claim of
our proposition in this case.

\smallskip

\noindent Case 2: assume that $j\in \Up_\ADG$. The proof is almost
the same as in the case of a down element. Let $\alpha$ be the
greatest down element smaller than $j$ and let $A_{\alpha}$ be the
vertex of $P$ labeled by $\alpha$. Then by construction of the
labeling, $A_{\alpha+1}$ is labeled by a larger number than $j$, and
$[A_{\alpha},A_{\alpha+1}]$ is the unique edge  of $P$ such that
$A_{\alpha}$ is labeled by a smaller number than $j$. Denote by
$\Lambda_\ADG$ the path from $A_l$ to $A_{\alpha}$ passing through
vertices of $P$ labeled by smaller numbers than $j$. This is the
path going from  $A_{\alpha}$ to $A_l$
 in clockwise direction on the boundary of $P$.

As above, $A_k\in \Lambda_\ADG$ if and only if the label of $A_k$ is
smaller than $j$. In other words, the path $\Lambda_\ADG$ consists
of all the vertices of $P$ labeled by smaller numbers than $j$.
 Therefore, again,  the cardinality of $\Lambda_\ADG$ is $j+1$.

Let $r_j^\ADG$ be the reflection mapping $A_\alpha$ to $A_0$ and
$A_{\alpha+1}$ to $A_{n+1}$. Recall that a reflection is an isometry
reversing the orientation of the plane. Then the path
$\Lambda_\ADG$, which is obtained by walking on the boundary of $P$
from $A_\alpha$ to $A_{l}$ in clockwise direction, is sent to the
path $\Lambda$ obtained by walking on the boundary of $P$ in
clockwise direction from $A_\alpha$ and going through
$j+1=|\Lambda_\ADG|$ vertices of $P$. Therefore
$\Lambda=\{A_0,A_1,\dots, A_j\}$. Hence $r_j^\ADG(A_l)$ is sent on
the final vertex of the path $\Lambda$ which is $A_j$,  proving the
first claim of our proposition.

\smallskip

Thus it remains to show that for a triangulation $T$ of $P$ we have
$\omega_j^\ADG(T)=\delta_j(r_j^\ADG\cdot T).$ We know that
$\Delta_j^\ADG(T)=A_k A_l A_m$ such that the label of $A_k$ is
smaller than $j$, which is smaller than the label of $A_m$. Write
 $A_a=r_j^\ADG(A_k)$ and $A_b=r_j^\ADG(A_m)$. Because of Proposition~\ref{prop:weight}, $a<j<b$ and therefore
$$
r_j^\ADG(\Delta_j^\ADG(T))= A_a A_jA_b=\Delta_j(r_j^\ADG\cdot T).
$$
So $(j-a)$ is the number of edges of $P$ between $A_l$ and $A_k$,
whose vertices are labeled by smaller numbers than $j$. Similarly,
$(b-j)$ is the number of edges between $A_l$ and $A_m$, whose
vertices  are labeled by smaller numbers than $j$, and $(b-j)$ is
the number of edges of $P$ between $A_l$ and $A_m$ and whose
vertices are labeled by larger numbers than $j$. So
$\omega_l^\ADG(T)=(j-a)(b-j)=\delta_j(r_j^\ADG\cdot T)$.
\end{proof}

\begin{cor}\label{cor:Canon} For any orientation $\ADG$ of the Coxeter graph of $S_n$ and for any
$j\in [n]$, we have
$$
\sum_{f\in \D_{n+2}} x^\ADG_j(f\cdot T) = (n+1)(n+2).
$$
\end{cor}
\begin{proof} Let $r_j^\ADG\in \mathcal D_{n+2}$ be as in Proposition~\ref{prop:weight}.

Suppose first that $j\in \Up_\ADG$, then
\begin{eqnarray*}
\sum_{f\in \D_{n+2}} x^\ADG_i(f\cdot T) &=&2(n+2)(n+1)-\sum_{f\in \D_{n+2}} \omega_i^\ADG(f\cdot T)\\
&=&2(n+2)(n+1)-\sum_{f\in \D_{n+2}} \delta_j(fr_j^\ADG\cdot T),\ \textrm{by Proposition~\ref{prop:weight}} \\
&=&2(n+2)(n+1)-\sum_{g\in \D_{n+2}} \delta_j(g^\ADG\cdot T),\ \textrm{since $r_j^\ADG\in\mathcal D_{n+2}$} \\
&=& (n+1)(n+2),\ \textrm{by Proposition~\ref{prop:canonique}}
\end{eqnarray*}
 If $i\in \Do_\ADG$,  the result follows from a similar calculation.
\end{proof}

\subsection{Center of gravity of associahedra}

\begin{thm}\label{thm:Main2} The center of gravity of $\Ass_\ADG(S_n)$ is $G=(\frac{n+1}{2},\frac{n+1}{2},\dots,\frac{n+1}{2})$ for any orientation $\ADG$.
 \end{thm}

By following precisely the same arguments as in
\S\ref{se:centergravity}, we just have to show the following
generalization of Theorem~\ref{thm:key}.

\begin{thm}\label{thm:keyGenAss} Let $\mathcal O$ be an orbit of the action of $\D_{n+2}$ on $\SOT_{n+2}$, then $G$ is the center of gravity of $\{M_\ADG(T)\,|\, T\in\mathcal O\}$. In particular, $\sum_{T\in\mathcal O} \vect{GM_\ADG(T)}=\vect 0. $
\end{thm}
\begin{proof} The proof is entirely similar to the proof of Theorem~\ref{thm:key}, using Corollary~\ref{cor:Canon} instead of Proposition~\ref{prop:canonique}.
\end{proof}

\section{Center of gravity of the cyclohedron}\label{se:3}

\subsection{The type $B$-permutahedron}
The hyperoctahedral  group  $W_n$ is defined by $W_n=\{\sigma\in S_{2n}\,|\, \sigma(i)+\sigma(2n+1-i)=2n+1,\ \forall i\in[n]\}$. The {\em type $B$-permutahedron} $\Perm(W_n)$ is the simple $n$-dimensional convex polytope defined as the convex hull of the points
$$
M(\sigma)=(\sigma(1),\sigma(2),\dots, \sigma (n))\in\mathbb R^{2n},\qquad \forall \sigma\in W_n.
$$
As $w_0=(2n,2n-1,\dots,3,2,1)\in W_n$, we deduce from the same
argument as in the case of $\Perm(S_n)$ that the center of gravity
of $\Perm(W_n)$ is
$$G=(\frac{2n+1}{2},\frac{2n+1}{2},\dots,\frac{2n+1}{2}).$$

\subsection{Realizations of the associahedron}
An orientation~$\ADG$ of~$\Gamma_{2n-1}$ is {\em symmetric} if the
edges $\{\tau_i,\tau_{i+1}\}$ and $\{\tau_{2n-i-1},\tau_{2n-i}\}$
are oriented in \emph{opposite directions} for all~$i\in [2n-2]$.
There is a bijection between symmetric orientations
of~$\Gamma_{2n-1}$ and orientations of the Coxeter graph of
$W_n$ (see \cite[\S1.2]{realisation1}). A triangulation $T\in \SOT_{2n+2}$ is {\em centrally
symmetric} if~$T$, viewed as a triangulation  of $P$, is centrally
symmetric. Let $\SOT_{2n+2}^B$ be the set of the centrally symmetric
triangulations of $P$. In \cite[Theorem
1.5]{realisation1} the authors show that for any symmetric
orientation $\ADG$ of $\Gamma_{2n-1}$. The convex hull
$\Ass_\ADG(W_{n})$ of $\{M_{\ADG}(T)\,|\,T\in \SOT^B_{2n+2}\}$ is a
realization of the cyclohedron with integer coordinates.

   Since the full orbit of symmetric triangulations under the action of $\D_{2n+2}$ on triangulations provides vertices of $\Ass_\ADG(W_{n})$, and vice-versa, Theorem~\ref{thm:keyGenAss} implies the following corollary.

 \begin{cor}\label{cor:Main} Let $\ADG$ be a symmetric orientation of $\Gamma_{2n-1}$, then the center of gravity of $\Ass_\ADG(W_n)$ is $G=(\frac{2n+1}{2},\frac{2n+1}{2},\dots,\frac{2n+1}{2})$.
 \end{cor}



\begin{thebibliography}{99}


\bibitem{bott_taubes}
{\sc R.~Bott and C.~Taubes}, {\em On the self-linking of knots}, J.
Math. Phys. {\bf 35} (1994), 5247--5287.

\bibitem{chapoton_fomin_zelevinsky}
{\sc F.~Chapoton, S.~Fomin and A.~Zelevinsky}, {\em Polytopal
realizations of generalized associahedra}, Canad. Math. Bull. {\bf
45} (2003), 537--566.

\bibitem{fomin_zelevinsky}
{\sc S.~Fomin and A.~Zelevinsky}, {\em Y-systems and generalized
associahedra},  Annals of Mathematics {\bf 158} (2003), 977--1018.

\bibitem{realisation1}
{\sc  C.~Hohlweg and C.~Lange},
  {\em Realizations of the Associahedron and Cyclohedron}, Discrete Comput Geom {\bf 37} (2007), 517--543.

\bibitem{realisation2}
  {\sc  C.~Hohlweg, C.~Lange, H.~Thomas},
  {\em Permutahedra and Generalized Associahedra}, \url{arXiv:math.CO/0709.4241}.


\bibitem{loday}
{\sc J.-L.~Loday}, {\em Realization of the Stasheff polytope},
Arch. Math. {\bf 83} (2004), 267--278.


\bibitem{shnider_sternberg}
{\sc S. Shnider and S. Sternberg}, {\em Quantum groups: {}From
coalgebas to Drinfeld algebras}, Graduate texts in mathematical
physics, International Press, 1994.

\bibitem{simion}
{\sc R.~Simion}, {\em A type-{B} associahedron}, Adv. Appl. Math.
{\bf 30} (2003), 2--25.


\bibitem{stasheff}
{\sc J.~Stasheff}, {\em Homotopy associativity of H-spaces I, II},
Trans. Amer. Math. Soc. {\bf 108} (1963), 275--312.


\bibitem{stasheff2}
{\sc J.~Stasheff}, {\em {}From operads to ``physically'' inspired
theories}, Operads: Proceedings of Renaissance Conferences
(Hartford, CT/Luminy, 1995), 53--81, Contemp. Math., 202, Americ.
Math. Soc. (1997).


\end{thebibliography}
\end{document}